\documentclass[11pt, a4paper]{article}

\usepackage{amsmath, amsfonts,amscd,flafter,epsf,amssymb,epsfig,subfigure,bm,url, mathrsfs, amsthm, comment}
\usepackage[margin=.9in]{geometry}

\usepackage{mathtools}

\usepackage{pgfplots,tikz}
\usepackage{xkeyval,tkz-base,tikz-3dplot}
\usepackage{tqft}
\usetikzlibrary{arrows, calc, shapes, automata, backgrounds, petri, positioning,fadings,decorations.pathreplacing, spy, intersections}
\usetikzlibrary{external}
\usetikzlibrary{fillbetween}
\usetikzlibrary{decorations.markings}
\usetikzlibrary{lindenmayersystems}
\tikzset{
	external/system call={
	xelatex \tikzexternalcheckshellescape
	-halt-on-error -interaction=batchmode --shell-escape --enable-write18
	-jobname "\image" "\texsource"}
}
\usepgfplotslibrary{external} 
\def\nextAngle{0}
\tikzset{
	next angle/.style={
		in=#1+180,
		out=\nextAngle,
		prefix after command= {\pgfextra{\def\nextAngle{#1}}}
	},
	start angle/.style={
		out=#1,
		nangle=#1,
	},
	nangle/.code={
		\def\nextAngle{#1}
	}
}

\usepackage{etoolbox}
\AtBeginEnvironment{tikzcd}{\tikzexternaldisable}
\AtEndEnvironment{tikzcd}{\tikzexternalenable}

\newtheorem{thm}{Theorem}[section]
\newtheorem{lemma}[thm]{Lemma}
\newtheorem{prop}[thm]{Proposition}
\newtheorem{cor}[thm]{Corollary}

\theoremstyle{definition}

\theoremstyle{remark}

\newcommand{\bbh}{\mathbb{H}}

\newcommand{\cald}{\mathcal{D}}
\newcommand{\caln}{\mathcal{N}}
\newcommand{\calp}{\mathcal{P}}
\newcommand{\calx}{\mathcal{X}}
\newcommand{\caly}{\mathcal{Y}}

\makeatletter \newcommand\dd[1]{\ensuremath{%
  \mathop{}\!\mathrm{d}#1\@ifnextchar\dd{\!}{}}}
\makeatother

\makeatletter
\DeclareFontFamily{U}{tipa}{}
\DeclareFontShape{U}{tipa}{m}{n}{<->tipa10}{}
\newcommand{\arc@char}{{\usefont{U}{tipa}{m}{n}\symbol{62}}}%

\newcommand{\arc}[1]{\mathpalette\arc@arc{#1}}

\newcommand{\arc@arc}[2]{%
  \sbox0{$\m@th#1#2$}%
  \vbox{
    \hbox{\resizebox{\wd0}{\height}{\arc@char}}
    \nointerlineskip
    \box0
  }%
}
\makeatother

\newcommand{\abs}[1]{\left| #1 \right|}
\newcommand{\gauss}[1]{\left[ #1 \right]}
\newcommand{\set}[1]{\left\{ #1 \right\}}
\newcommand{\bracket}[1]{\left( #1 \right)}

\author{
  LastName1, FirstName1\\
  \texttt{first1.last1@xxxxx.com}
  \and
  LastName2, FirstName2\\
  \texttt{first2.last2@xxxxx.com}
}

\author{
	Wujie Shen\\
	\texttt{1800010683@pku.edu.cn}
	\and 
	Jiajun Wang\\
	\texttt{wjiajun@pku.edu.cn}
}

\date{%
	LMAM, School of Mathematical Sciences, Peking University\\%
	Bejing, 100871, P. R. China
	\\[2ex]%
	\today
}

\title{Minimal length of nonsimple closed geodesics on hyperbolic surfaces}

\begin{document}

\maketitle

\begin{abstract}
In the present paper, we show that the minimal length of closed geodesics on finite-type hyperbolic surfaces with self-intersection number $k$ has order $2\log k$ as $k$ gets large.
\end{abstract}

\section{Introduction}

The length of a simple closed geodesic on a hyperbolic surface can be arbitrarily small and the collar lemma states that it has a collar neighborhood depending only on the length of the geodesic. On the other hand, Hempel showed in \cite{H2797} that a nonsimple closed geodesic has a universal lower bound $2\log(1+\sqrt{2})$ and Yamada showed in \cite{Y2799} that $2 \cosh^{-1}(3)=4\log(1+\sqrt{2})$ is the best possible lower bound and is attained on a pair of pants with ideal punctures. Basmajian showed in \cite{B2794} that a nonsimple closed geodesic has a similar \emph{stable neighborhood}, and the length of a closed geodesic gets arbitrarily large as its self-intersection number gets large (\cite[Corollary 1.2]{B2794}). In \cite{B2795}, Baribaud computed the minimal length of geodesics with given self-intersection number on pairs of pants.

It is interesting to study the growth rate of the minimal length of non-simple closed geodesics with respect to its self-intersection number $k$ (called a \emph{$k$-geodesic} for brevity). Let $\omega$ be a closed geodesic or a geodesic segment on a hyperbolic surface. Denote its length by $\ell(\omega)$ and its self-intersection number by $\abs{\omega\cap\omega}$. $\abs{\omega\cap\omega}$ counts the intersection points of $\omega$ with multiplicity that an intersection point with $n$ preimages contribute $\binom{n}{2}$ to $\abs{\omega\cap\omega}$. (See \S \ref{sec:estimate} for definitions.)

On a fixed hyperbolic surface, Basmajian showed in \cite{B2803} that a $k$-geodesic has length no less than $C\sqrt{k}$, where $C$ is a constant depending only on the hyperbolic surface. Erlandsson and Parlier showed in \cite{EP2802} that the self-intersection number of the shortest geodesic with at least $k$ self-intersections on a hyperbolic surface is bounded by a function depending only on $k$ that grows linearly as $k$ gets large. 

Let $M_k$ be the infimum of lengths of $k$-geodesics among all hyperbolic surfaces. Basmajian showed (\cite[Corollary 1.4]{B2803}) that
\begin{equation}\label{eqn:growth_bas}
\tfrac12\log \frac{k}2\leqslant M_k \leqslant 2\cosh^{-1}(2k+1)\asymp 2\log k
\end{equation}
The notation $f(k)\asymp g(k)$ means that ${f(k)}/{g(k)}$ is bounded from above and below by positive constants. 
Basmajian also showed that $M_k$ is realized by a $k$-geodesic on some hyperbolic surface. 

\begin{thm}\label{thm:main}
Let $M_k$ be the infimum of lengths of $k$-geodesics among all oriented, metrically complete, finite-volume hyperbolic surfaces, then
\begin{equation}\label{eqn:length_asymptotic}
\lim_{k\to\infty} \frac{M_k}{2\log k}=1
\end{equation}
\end{thm}

Theorem \ref{thm:main} can be generalized to general orientable finite-type hyperbolic surfaces, possibly with holes and geodesic boundaries, since they can be doubled to get a surface as in Theorem \ref{thm:main}.

The proof of Theorem \ref{thm:main} benefits from the following upper bound of $M_k$ by Baribaud in \cite{B2795}

\begin{thm}[Baribaud, \cite{B2795}]\label{thm:baribaud}
For any $\varepsilon>0$, $M_k<(2+\varepsilon)\log k$ for $k$ sufficiently large.
\end{thm}

We study the self-intersection number of a closed geodesic with given length.

\begin{thm}\label{thm:bound_intersection_by_length}
The self-intersection number of a closed geodesic on a hyperbolic surface with length $L$ is less than $9L^2e^{L/2}$.
\end{thm}

From the proof of Theorem \ref{thm:bound_intersection_by_length} and Lemma \ref{lemma:inj_radius_outside_nbhds} we actually have
$$\varlimsup_{L\to+\infty}\frac{\abs{\Gamma\cap\Gamma}}{L^2e^{L/2}}\leqslant8$$

Theorem \ref{thm:main} follows from Theorem \ref{thm:baribaud} and the following lower bound estimate of $M_k$.

\begin{cor}\label{thm:lower_bound}
For any $\varepsilon>0$, $M_k>(2-\varepsilon)\log{k}$ for $k$ sufficiently large.
\end{cor}

Theorem \ref{thm:bound_intersection_by_length} is obtained by estimating the self-intersection number in certain thin and thick parts of a hyperbolic surface for a possibly nonsimple closed geodesic.

\subsubsection*{Acknowledgement}

The present work is motivated by a seminar talk on ``simple closed curves on surfaces" given by Yi Huang. We would like to thank Professors Shicheng Wang and Yi Liu for helpful discussions. The second author is partly supported by NSFC 12131009 and NKRDPC 2020YFA0712800.

\section{Neighborhoods of sufficiently short geodesics and cusps}\label{sec:nbhds}

Let $L\geqslant4\log(1+\sqrt{2})>3.5$ be a constant. Let $\Sigma$ be an oriented, metrically complete hyperbolic surface of finite type. Topologically $\Sigma$ is an orientable surface of genus $g$ with $n$ punctures such that $2g+n\geqslant 3$. Denote the length of a curve $c$ on $\Sigma$ by $\ell(c)$. 

\subsection{Neighborhoods of short geodesics}

Let $\mathcal{X}$ be the set of closed geodesics $c$ on $\Sigma$ with length
\begin{equation}\label{eqn:short_length}
\ell(c)<e^{-{L}/4}\leqslant e^{-\log(1+\sqrt{2})}=\sqrt{2}-1<0.42. 
\end{equation}
Curves in $\mathcal{X}$ are simple by Hempel's universal lower bounds for nonsimple closed geodesics. The collar lemma (\cite[Lemma 13.6]{FM2012}) states that, for any $c$ in $\mathcal{X}$, $N(c)=\set{p\in\Sigma\,:\,d(p,c)\leqslant w(\ell(c))}$ is an embedded annulus, where $w(x)$ is defined by
$$w(x):=\sinh^{-1}\bracket{\frac{1}{\sinh({x}/{2})}}$$

\begin{lemma}\label{lemma:w_estimate}
For each $x>0$, we have
\begin{equation}
w(x)>\log{\frac{4}{x}}
\end{equation}
\end{lemma}

\begin{proof}
Since $\sinh(w(x))\sinh({x}/2)=1$,
$$\frac{e^{w(x)}-e^{-w(x)}}{2}\cdot\frac{e^{x/2}-e^{-x/2}}{2}=1,$$
$$e^{w(x)}=\frac{2}{e^{x/2}-e^{-x/2}}+\sqrt{\bracket{\frac{2}{e^{x/2}-e^{-x/2}}}^2+1}=\frac{2+e^{x/2}+e^{-x/2}}{e^{x/2}-e^{-x/2}}=\frac{e^{x/4}+e^{-x/4}}{e^{x/4}-e^{-x/4}}>\frac{4}{x},$$
	The last inequality uses the fact that $\displaystyle{\frac{e^t+e^{-t}}{e^t-e^{-t}}>\frac{1}{t}}$ for $t>0$.
\end{proof}

It follows from Lemma \ref{lemma:w_estimate} and \eqref{eqn:short_length} that for any $c\in\mathcal{X}$, 
\begin{equation}\label{eqn:collar_length}
w(\ell(c))>\log\bracket{{4}/{\ell(c)}}>\log(4(\sqrt{2}+1))>2.26.
\end{equation}

\begin{lemma}
Geodesics in $\mathcal{X}$ are disjoint. 
\end{lemma} 

\begin{proof} Suppose that two simple closed geodesics $c_1,c_2\in\mathcal{X}$ intersects. Since $\ell(c_2)<0.42$ by \eqref{eqn:short_length} and $w(\ell(c_1))>2.26$ by \eqref{eqn:collar_length}, we have $c_2\subseteq N(c_1)$. $N(c_1)$ is a hyperbolic annulus and the only simple closed geodesic in $N(c_1)$ is $c_1$ itself. It follows that $c_1=c_2$.
\end{proof}

Suppose that $\mathcal{X}=\set{c_1,\cdots,c_m}$ ($\mathcal{X}$ could be empty). For each $1\leqslant i\leqslant m$, let
\begin{equation}\begin{aligned}
N_3(c_i)=&\set{x\in{\Sigma}:d(x,c_i)<{\log{\frac{1}{\ell(c_i)}}}},\\
N_2(c_i)=&\set{x\in{\Sigma}:d(x,c_i)<{\log{\frac{1}{\ell(c_i)}}-\frac{L}{4}}},\\
N_1(c_i)=&\set{x\in{\Sigma}:d(x,c_i)<{\log{\frac{1}{\ell(c_i)}}-\frac{L}{2}}},\\
\end{aligned}\end{equation}
Each $N_j(c_i)$ is either empty or an annulus, as illustrated in Figure \ref{fig:nbhd_short_geodesics}.
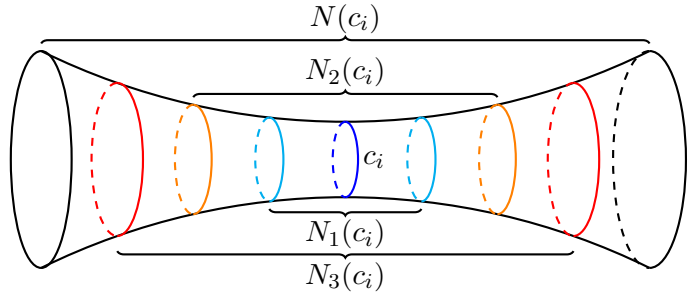
\begin{figure}[htbp]
\begin{center}
\tikzexternaldisable
\begin{tikzpicture}[declare function={
	R = 9;
	f(\x) = R+.5-sqrt(R*R-\x*\x);
	ratio = 3;
}]
\draw [thick] (-4,{f(4)}) arc ({270-asin(4/R)}:{270+asin(4/R)}:{R} and {R});
\draw [thick] (-4,{-f(4)}) arc ({90+asin(4/R)}:{90-asin(4/R)}:{R} and {R});

\draw [thick, color=black] (-4,0) ellipse ({0.4} and {f(4)});

\foreach \a/\b in {-3/red, -2/orange, -1/cyan, 0/blue, 1/cyan, 2/orange, 3/red} {
	\draw [color=\b, thick] (\a,{-f(\a)}) arc (-90:90:{f(\a)/ratio} and {f(\a)});
	\draw [dashed, color=\b, thick] (\a,{f(\a)}) arc (90:270:{f(\a)/ratio} and {f(\a)});
}
\draw [color=black, thick] (4,{-f(4)}) arc (-90:90:{f(4)/ratio} and {f(4)});
\draw [color=black, thick, dashed] (4,{f(4)}) arc (90:270:{f(4)/ratio} and {f(4)});

\draw [thick, decoration={
	brace, mirror, raise=0.1cm
	}, decorate
] (-1, {-f(1)}) -- (1,{-f(1)}) 
	node [pos=0.5,anchor=north,yshift=-0.1cm] {$N_1(c_i)$}; 

\draw [thick, decoration={
	brace, mirror, raise=0.2cm
	}, decorate
] (-3, {-f(3)}) -- (3,{-f(3)}) 
	node [pos=0.5,anchor=north,yshift=-0.2cm] {$N_3(c_i)$}; 

\draw [thick, decoration={
	brace, raise=0.1cm
	}, decorate
] (-2,{f(2)}) -- (2,{f(2)}) 
	node [pos=0.5,anchor=south,yshift=0.1cm] {$N_2(c_i)$}; 

\draw [thick, decoration={
	brace, raise=0.1cm
	}, decorate
] (-4,{f(4)}) -- (4,{f(4)}) 
	node [pos=0.5,anchor=south,yshift=0.1cm] {$N(c_i)$}; 
  			
\path ({f(0)/ratio},0) node[circle, inner sep=1pt, label={[label distance=.4em, anchor=center]0:$c_i$}]{};
  			
\end{tikzpicture}
\end{center}
\caption{\label{fig:nbhd_short_geodesics}
Neighborhoods of short geodesics.}
\end{figure}

\subsection{Neighborhoods of cusps}

When $\Sigma$ has punctures, we consider the universal covering $p:\bbh^2\to\Sigma$, where $\bbh^2$ is the hyperbolic plane. Each puncture has a neighborhood whose boundary lifts to a union of horocycles that can intersect in at most points of tangency. Such a neighborhood is called a \emph{cusp} of the surface.

In the upper half-plane model for $\bbh^2$. Let $\Gamma$ be a cyclic group generated by a parabolic isometry of $\bbh^2$ fixing the point $\infty$. Let $H_c=\set{(x,y)\in\bbh^2\,\big|\,y\geqslant c}$ be a horoball. Each cusp can be modelled as $H_c/\Gamma$ for some $c$ up to isometry, and is diffeomorphic to $S^1\times[c,\infty)$ so that each circle $S^1\times\set{t}$ with $t\geqslant c$ is the image of a horocycle under $p$. Each circle is also called a \emph{horocycle} by abuse of notation. The circle $S^1\times\set{t}$ with $t\geqslant c$ is called an \emph{Euclidean circle}. A cusp is \emph{maximal} if it lifts to a union of horocycles with disjoint interiors such that there exists at least one point of tangency between different horocycles.

\begin{lemma}[Adams, \cite{A2817}]\label{lemma:cusp_maximal_area}
For an orientable, metrically complete hyperbolic surface, the area with a maximal cusp is at least $4$. The lower bound $4$ is realized only in an ideal pair of pants.
\end{lemma}


Let $\caly$ be the set of punctures of $\Sigma$. Each puncture $c_i$ of $\Sigma$ has a maximal cusp whose boundary Euclidean circle $c$ has 
$\ell(c)\geqslant4$. The cusp of area $4$ that can be lifted to $\bbh^2$, as in Figure \ref{fig:nbhd_cusps},
\begin{figure}[htbp]
\begin{center}
\tikzexternaldisable
\begin{tikzpicture}
\def\x{1.5}

\draw[thick] (-1.5*\x,0) to (-1.5*\x,6*\x);
\draw[thick] (1.5*\x,0) to (1.5*\x,6*\x);

\fill[cyan!20] (-1.5*\x,4*\x) rectangle (1.5*\x,6*\x);

\draw[thick] (-1.5*\x, \x) to (1.5*\x,\x);

\draw[thick, color=red] (-1.5*\x, 2*\x) to (1.5*\x,2*\x);
\draw[thick, color=orange] (-1.5*\x, 3*\x) to (1.5*\x,3*\x);
\draw[thick, color=cyan] (-1.5*\x, 4*\x) to (1.5*\x,4*\x);

\draw[dashed] (-3*\x,0) to (3*\x,0);

\path (-1.5*\x,\x) node[circle, fill, inner sep=1.5pt, label=left:{$P$}]{};
\path (1.5*\x,\x)  node[circle, fill, inner sep=1.5pt, label=right:{$Q$}]{};

\path (-1.5*\x,2*\x) node[circle, fill, inner sep=1.5pt, label=left:{$A_3$}]{};
\path (1.5*\x,2*\x)  node[circle, fill, inner sep=1.5pt, label=right:{$B_3$}]{};

\path (-1.5*\x,3*\x) node[circle, fill, inner sep=1.5pt, label=left:{$A_2$}]{};
\path (1.5*\x,3*\x)  node[circle, fill, inner sep=1.5pt, label=right:{$B_2$}]{};

\path (-1.5*\x,4*\x) node[circle, fill, inner sep=1.5pt, label=left:{$A_1$}]{};
\path (1.5*\x,4*\x)  node[circle, fill, inner sep=1.5pt, label=right:{$B_1$}]{};

\draw [thick, decoration={
	brace, raise=0.1cm
	}, decorate
] (-2*\x, 2*\x) -- (-2*\x, 6*\x) 
	node [pos=0.5,anchor=east,xshift=-0.1cm] {$N_3(c_i)$}; 

\draw [thick, decoration={
	brace, raise=0.1cm
	}, decorate
] (2*\x, 6*\x) -- (2*\x, 3*\x)
	node [pos=0.5,anchor=west,xshift=0.1cm] {$N_2(c_i)$}; 

\draw (0,5*\x) node {$N_1(c_i)$};

\draw (0,6*\x) node[above] {$\infty$};

	  			
\end{tikzpicture}
\end{center}
\tikzexternalenable
\caption{\label{fig:nbhd_cusps}
Neighborhoods of cusps.}
\end{figure}
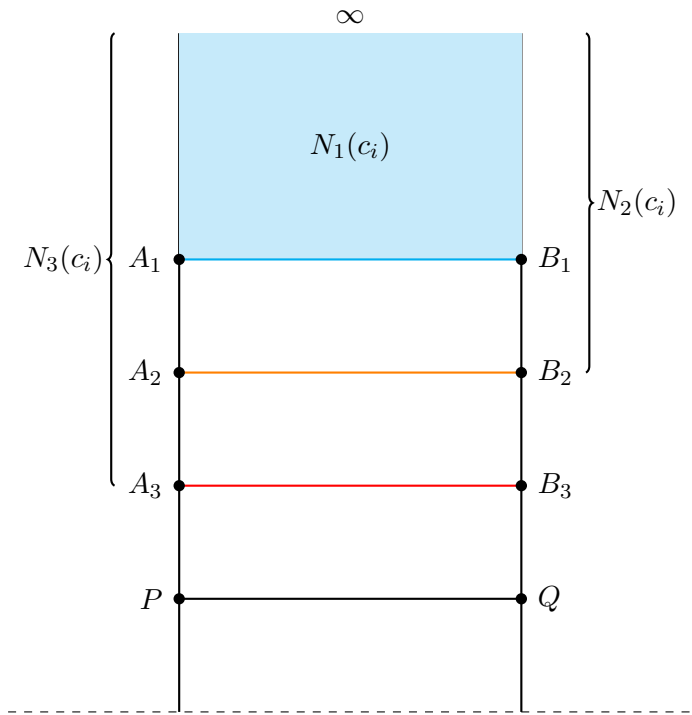
such that $p$ maps the the triangle $\infty PQ$ to the cusp of area $4$ at $c_i$, and $p$ maps the interior of the triangle homeomorphically. Let $N(c_i)$ be the interior of the cusp of area 4. Choose points $A_1, A_2, A_3$ on the ray from $P$ to $\infty$, and $B_1, B_2, B_3$ on the ray from $Q$ to $\infty$ so that
$$d(P,A_3)=d(Q,B_3)=2\log2,\qquad d(A_3,A_2)=d(A_2,A_1)=d(B_3,B_2)=d(B_2,B_1)=\frac{L}4$$
For $j=1,2,3$, let $N_j(c_i)$ be the interior of the image of the triangle $\infty A_jB_j$ under $p$. In other words, the boundary $\partial N_3(c_i)$, $\partial N_2(c_i)$ and $\partial N_1(c_i)$ are Euclidean circles of length 1, $e^{-\frac{L}4}$,  and $e^{-\frac{L}2}$ respectively.

\subsection{A thick-thin decomposition}

Let $\caln:=\bigcup_{c\in\calx\cup\caly}N_2(c)$. $\caln$ is called the \emph{thin part} of $\Sigma$ while $\Sigma\setminus\caln$ is the \emph{thick part} of $\Sigma$. Note that our thick-thin decomposition is different from the regular thick-thin decomposition and it depends on the constant $L$. We will estimate the injectivity radius at a point in the thick part.

\begin{lemma}\label{lemma:ends_annulus_length}
Let $A$ be an annulus in $\Sigma$ with boundary circles $\gamma_1$ and $\gamma_2$, where $\gamma_1$ is a geodesic and $\gamma_2$ is piecewise smooth. If there exists $x\in \gamma_2$ such that $d(\gamma_1, x)=d>0$, then
\begin{equation}
\sinh\bracket{\frac{\ell(\gamma_2)}2}>\frac14 e^d\ell(\gamma_1)
\end{equation}
If we further have $\ell(\gamma_1)<e^{-d}$, then
\begin{equation}
\ell(\gamma_2)>\frac{12}{25}e^{d}\ell(\gamma_1).
\end{equation}
\end{lemma}

\begin{proof}
The universal covering $p:\bbh^2\to\Sigma$ from the Poincar\'e disk $\bbh^2$ to $\Sigma$ is locally isometric. Let $\widetilde{\gamma}_1$ be a lift of $\gamma_1$. The connected component $\widetilde{A}$ of $p^{-1}(A)$ containing $\widetilde{\gamma}_1$ is a universal cover of the annulus $A$. Let $\widetilde{x}_1$ and $\widetilde{x}_2$ be two adjacent lifts of $x$ in $\widetilde{A}$ so that the boundary of $\widetilde{A}$ betweeen $\widetilde{x}_1$ and $\widetilde{x}_2$ is a lift of $\gamma_2\setminus{x}$. Let $\widetilde{\eta}_1$ and $\widetilde{\eta}_2$ be the shortest geodesics from $\widetilde{x}_1$ and $\widetilde{x}_2$ to $\widetilde{\gamma}_1$ respectively. Then $\eta:=p\circ\widetilde{\eta}_1=p\circ\widetilde{\eta}_2$ is a geodesic connecting $x$ and $\gamma_1$ and $\widetilde{\eta}_1$ and $\widetilde{\eta}_2$ are both perpendicular to $\gamma_1$. Let $\widetilde{y}_1$ and $\widetilde{y}_2$ be the two feet. Without loss of generality, we may assume $\widetilde{\gamma}_1$ is the horizontal diameter of $\bbh^2$ and the origin $O$ is the middle point of $\widetilde{y}_1$ and $\widetilde{y}_2$, as illustrated in Figure \ref{fig:boundary_length_annulus}.
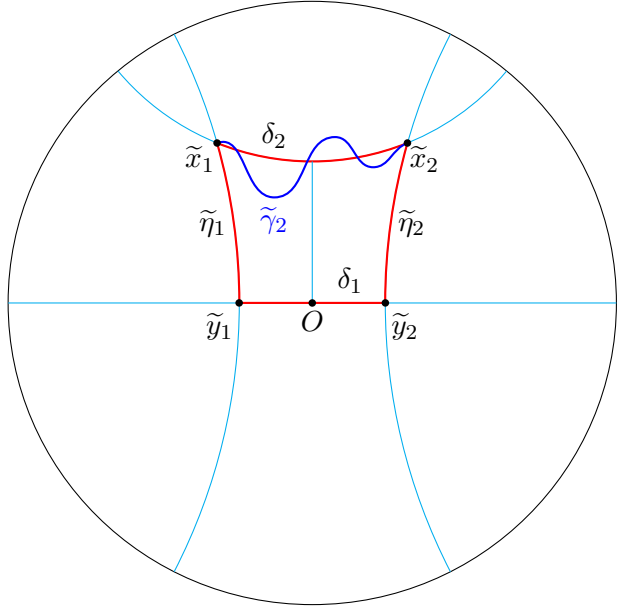
\begin{figure}[htbp]
\centering
\tikzexternaldisable
\begin{tikzpicture}
\def\x{4}
\coordinate (O) at (0,0);
\coordinate (P) at (-2.2*\x,0);
\coordinate (Q) at (2.2*\x,0);
\coordinate (K) at (0, 1.3*\x);

\begin{pgfinterruptboundingbox}
\path[name path=circle_O] (O) circle (\x);
\path[name path=PQ] (P) to (Q);
\coordinate[name intersections={of=circle_O and PQ, by={N,M}}];

\path[name path=circle_P] let
  \p1 = ($ (M) - (P) $),  \n1 = {veclen(\x1,\y1)},
  \p2 = ($ (N) - (P) $),  \n2 = {veclen(\x2,\y2)} in
  (P) circle ({sqrt(\n1)*sqrt(\n2)});
\coordinate[name intersections={of=circle_P and PQ, by={A}}];
\coordinate[name intersections={of=circle_P and circle_O, by={G,S}}];

\path[name path=circle_Q] let
  \p1 = ($ (M) - (Q) $),  \n1 = {veclen(\x1,\y1)},
  \p2 = ($ (N) - (Q) $),  \n2 = {veclen(\x2,\y2)} in
  (Q) circle ({sqrt(\n1)*sqrt(\n2)});
\coordinate[name intersections={of=circle_Q and PQ, by={B}}];
\coordinate[name intersections={of=circle_Q and circle_O, by={H,T}}];

\path[name path=circle_K] (K) circle ({sqrt(0.3*2.3)*\x});
\coordinate[name intersections={of=circle_P and circle_K, by={X,C}}];
\coordinate[name intersections={of=circle_Q and circle_K, by={X,D}}];
\coordinate[name intersections={of=circle_O and circle_K, by={E,F}}];

\end{pgfinterruptboundingbox}

\draw (O) circle (\x);
\draw[color=cyan, thin] (180:\x) to (0:\x);

\pgfmathanglebetweenpoints{\pgfpointanchor{P}{center}}{\pgfpointanchor{G}{center}}
\pgfmathsetmacro{\anglePG}{\pgfmathresult}
\draw[name path=arc_GS, color=cyan, thin]
  let \p1 = ($ (G) - (P) $),  \n1 = {veclen(\x1,\y1)} in
  (G) arc (\anglePG:-\anglePG:\n1);
\draw[name path=HT, color=cyan, thin]
  let \p1 = ($ (G) - (P) $),  \n1 = {veclen(\x1,\y1)} in
  (H) arc ({180-\anglePG}:{180+\anglePG}:\n1);
 
\pgfmathanglebetweenpoints{\pgfpointanchor{K}{center}}{\pgfpointanchor{E}{center}}
\pgfmathsetmacro{\angleKE}{\pgfmathresult}
\draw[name path=arc_EF, color=cyan, thin]
  (E) arc (\angleKE:540-\angleKE:{sqrt(0.3*2.3)*\x});

\pgfmathanglebetweenpoints{\pgfpointanchor{K}{center}}{\pgfpointanchor{C}{center}}
\pgfmathsetmacro{\angleKC}{\pgfmathresult}
\draw[name path=arc_CD, color=red, thick]
  (C) arc (\angleKC:540-\angleKC:{sqrt(0.3*2.3)*\x});

\pgfmathanglebetweenpoints{\pgfpointanchor{P}{center}}{\pgfpointanchor{C}{center}}
\pgfmathsetmacro{\anglePC}{\pgfmathresult}
\draw[name path=arc_AC, color=red, thick]
  let \p1 = ($ (C) - (P) $),  \n1 = {veclen(\x1,\y1)} in
  (C) arc (\anglePC:0:\n1);
\draw[name path=arc_BD, color=red, thick]
  let \p1 = ($ (G) - (P) $),  \n1 = {veclen(\x1,\y1)} in
  (D) arc ({180-\anglePC}:180:\n1);

\draw[color=red, thick] (A) to (B);
\draw[color=cyan] (O) to (0,{(1.3-sqrt(0.3*2.3))*\x});

\draw[smooth, color=blue, thick] (C)
  to [start angle=20, next angle=0] (-0.5,1.4)
	to [next angle=0] (0.3,2.2)
	to [next angle=0] (0.8,1.8)
	to [next angle=20] (D);

\node[above] at (0.5,0) {$\delta_1$};
\node[above=0.1, anchor=center] at ($(C)!0.3!(D)$) {$\delta_2$};
\node[left=0.2, anchor=center] at ($(A)!0.5!(C)$) {$\widetilde{\eta}_1$};
\node[right=0.2, anchor=center] at ($(B)!0.5!(D)$) {$\widetilde{\eta}_2$};

\node[below=0.3, anchor=center, color=blue] at (-0.5,1.4) {$\widetilde{\gamma}_2$};

\path (A) node[circle, fill, inner sep=1pt, label={[shift={(235:0.45)}, anchor=center]{$\widetilde{y}_1$}}]{};
\path (B) node[circle, fill, inner sep=1pt, label={[shift={(305:0.45)}, anchor=center]{$\widetilde{y}_2$}}]{};
\path (C) node[circle, fill, inner sep=1pt, label={[shift={(230:0.35)}, anchor=center]{$\widetilde{x}_1$}}]{};
\path (D) node[circle, fill, inner sep=1pt, label={[shift={(310:0.35)}, anchor=center]{$\widetilde{x}_2$}}]{};

\path (O) node[circle, fill, inner sep=1pt, label={[shift={(270:0.3)}, anchor=center]{$O$}}]{};

\end{tikzpicture}
\caption{\label{fig:boundary_length_annulus}
A covering of the annulus.}
\end{figure}

The geodesic $\delta_2$ between $\widetilde{x}_1$ and $\widetilde{x}_2$, $\widetilde{\eta}_1$, $\widetilde{\eta}_2$ and the geodesic $\delta_1$ between $\widetilde{y}_1$ and $\widetilde{y}_2$ form a Saccheri quadrilateral, and half of it is a Lambert quadrilateral. The property of the Lambert quadrilateral gives
\begin{equation*}\label{eqn:ends_annulus_length_1}
\sinh\bracket{\frac{\ell(\delta_2)}2}=\sinh\bracket{\frac{\ell\bracket{\delta_1}}2}\cosh\bracket{\ell\bracket{\widetilde{\eta}_1}}
\end{equation*}
We have $\ell(\gamma_2)\geqslant\ell(\delta_2)$, $\ell(\delta_1)=\ell\bracket{\gamma_1}$ and $\ell\bracket{\widetilde{\eta}_1}=\ell\bracket{\widetilde{\eta}_2}=d$. Hence
\begin{equation*}\label{eqn:ends_annulus_length_2}
\sinh\bracket{\frac{\ell(\gamma_2)}2}\geqslant\sinh\bracket{\frac{\ell(\gamma_1)}2}\cosh(d)
>\frac14 e^d\ell(\gamma_1)
\end{equation*}
When $\ell(\gamma_1)<e^{-d}$, we have $e^d\ell(\gamma_1)<1$. Since $\displaystyle{\sinh^{-1}(t)>\frac{24}{25}t}$ for $0<t<\frac14$, we have
$$\ell(\gamma_2)>2\sinh^{-1}\bracket{\frac14e^d\ell(\gamma_1)}>2\cdot\frac{24}{25}\cdot \frac14e^d\ell(\gamma_1)=\frac{12}{25}e^d\ell(\gamma_1)$$
\end{proof}

The following lemma estimates the injectivity radius for points in the neighborhood of cusps.

\begin{lemma}\label{lemma:injectivity_radius_cusp}
Let $x\in{N_3(c_i)}$ for a cusp $c_i$ in $\mathcal{Y}$. If $d(x,\partial{N(c_i)})=d$, then the injective radius at $x$ is $\sinh^{-1}(2e^{-d})$.
\end{lemma}

\begin{proof}
In the universal cover of the cusp $c_i$, as in Figure \ref{fig:cusp_injectivity_radius}, 
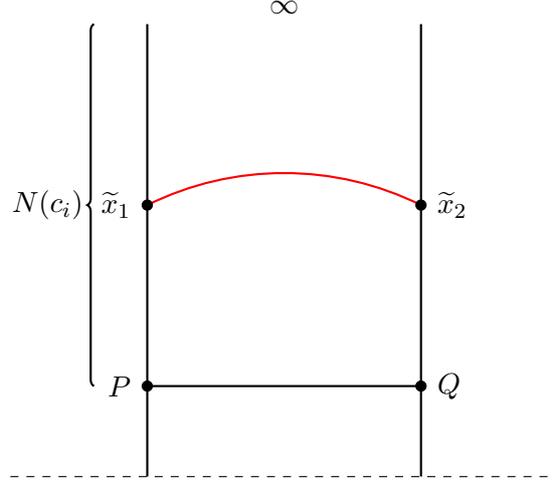
\begin{figure}[htbp]
\centering
\tikzexternaldisable
\begin{tikzpicture}
\def\x{1.2}

\coordinate (X1) at (-1.5*\x,3*\x);
\coordinate (X2) at (1.5*\x,3*\x);
\coordinate (O) at (0,0);

\pgfmathanglebetweenpoints{\pgfpointanchor{O}{center}}{\pgfpointanchor{X2}{center}}
\pgfmathsetmacro{\angleOX}{\pgfmathresult}
\draw[name path=circle_x_x, color=red, thick]
  let \p1 = ($ (X1) - (O) $),  \n1 = {veclen(\x1,\y1)} in
  (X2) arc (\angleOX:{180-\angleOX}:\n1);

\draw[thick] (-1.5*\x,0) to (-1.5*\x,5*\x);
\draw[thick] (1.5*\x,0) to (1.5*\x,5*\x);


\draw[thick] (-1.5*\x, \x) to (1.5*\x,\x);


\draw[dashed] (-3*\x,0) to (3*\x,0);

\path (-1.5*\x,\x) node[circle, fill, inner sep=1.5pt, label=left:{$P$}]{};
\path (1.5*\x,\x)  node[circle, fill, inner sep=1.5pt, label=right:{$Q$}]{};


\path (-1.5*\x,3*\x) node[circle, fill, inner sep=1.5pt, label=left:{$\widetilde{x}_1$}]{};
\path (1.5*\x,3*\x)  node[circle, fill, inner sep=1.5pt, label=right:{$\widetilde{x}_2$}]{};

\draw [thick, decoration={
	brace, raise=0.1cm
	}, decorate
] (-2*\x, 1*\x) -- (-2*\x, 5*\x) 
	node [pos=0.5,anchor=east,xshift=-0.1cm] {$N(c_i)$}; 

\draw (0,5*\x) node[above] {$\infty$};

\end{tikzpicture}
\tikzexternalenable
\caption{\label{fig:cusp_injectivity_radius}
Injectivity radius near a cusp.}
\end{figure}
let $PQ$ be a lift of $\partial N(c_i)$, $\widetilde{x}_1$ and $\widetilde{x}_2$ be two lifts of $x$ and $\gamma$ be the geodesic connecting them. Without loss of generality, we may assume that $P=(-2,1)$ and $Q=(2,1)$. Since $d(\widetilde{x}_1, P)=d$, $\widetilde{x}_1=(-2, e^d)$. Hence we have
$$2\sinh^2\bracket{\frac{\ell(\gamma)}2}=\frac{\bracket{2-(-2)}^2}{2 (e^d)^2}$$
Hence
$$\sinh\bracket{\frac{\ell(\gamma)}2}=\frac{2}{e^d}$$

\end{proof}

\begin{lemma}\label{lemma:nbhd_N3_disjoint}
For any two distinct $c_i,c_j\in\calx\cup\caly$, we have $N_3(c_i)\cap N_3(c_j)=\emptyset$.
\end{lemma}  

\begin{proof} We claim that
\begin{equation}\label{eqn:lemma_nbhd_N3_disjoint_1}
d(c_i,c_j)>\log{\frac{2}{\ell(c_i)}}=\log\frac{1}{\ell(c_i)}+\log2
\end{equation}
If $c_i$ and/or $c_j$ are punctures, $d(c_i,c_j)=\infty$ and the claim is trivial. Otherwise, there exists $x\in c_j$ that $d(x,c_i)\leqslant\log{\frac{2}{\ell(c_i)}}$. Since $\ell(c_j)<0.42<\log2$, for any point $y$ in $c_j$, we have
$$d(y,c_i)\leqslant d(y,x)+d(x,c_i)<\log\frac{2}{\ell(c_i)}+\log2=\log\frac{4}{\ell(c_i)}<w(\ell(c_i))$$
Hence $c_j\subseteq{N(c_i)}$ and this contradicts the uniqueness of closed geodesics in $N(c_i)$. The claim follows.

Suppose that $N_3(c_i)\cap{N_3(c_j)}\neq\emptyset$.	
\begin{enumerate}
\item If $c_i$ and $c_j$ are both curves in $\calx$, we have
\begin{equation}
\max\set{\log{\frac{2}{\ell(c_i)}},\log{\frac{2}{\ell(c_j)}}}<{d(c_i,c_j)}<{\log{\frac{1}{\ell(c_i)}}+\log{\frac{1}{\ell(c_j)}}}
\end{equation}
Let $\gamma$ be (one of) the shortest geodesic(s) connecting $c_i$ and $c_j$. Then $\ell(\gamma)=d(c_i,c_j)$ and there exists a point $x_0$ on $\gamma$ such that 
$$d(x_0,c_i)<{\log{\frac{1}{\ell(c_i)}}-\log2},\qquad d(x_0,c_j)=\log{\frac{1}{\ell(c_j)}}+\log2.$$
Hence $x_0\in{N(c_i)\cap{N(c_j)}}$. 



Let $\gamma$ the shortest nontrivial closed curve passes $x_0$ freely homotopic to $c_i$. $\gamma$ is locally geodesic at $\gamma\setminus\set{x_0}$. In $N(c_i)$, using the formula in Lambert quadrilteral we get
$$\ell(\gamma)\leqslant\sinh{\ell(\gamma)}=\sinh{\ell(c_i)}\cosh{d(x_0,c_i)}<{\sinh{\ell(c_i)}\cosh\bracket{{\log{\frac{1}{\ell(c_i)}}-\log2}}}\leqslant{\frac{e^{\ell(c_i)}-e^{-\ell(c_i)}}{4\ell(c_i)}}<\frac{3}{5}$$
Here we use the fact that $\ell(c_i)<\frac12$ and $\sinh{t}<\frac{6}{5}t$ for $0<t<\frac{1}{2}$. 

For any $y\in\gamma$, we have
$$d(y, c_j)\leqslant d(y,x_0)+d(x_0,c_j)\leqslant \frac3{10}+\log\frac1{\ell(c_j)}+\log2<\log\frac4{\ell(c_j)}<w(\ell(c_j))$$
Hence $\gamma\subset N(c_j)$.
By Lemma \ref{lemma:ends_annulus_length}, we get
	$$\frac3{5}>\ell(\gamma)>\frac{12}{25}{\ell(c_j)\exp\bracket{\log{\frac{2}{\ell(c_j)}}}}=\frac{24}{25}$$
The lemma follows from the contradiction.

\item If $c_i$ is a puncture in $\caly$. The boundary curve $\gamma_i$ of $N_3(c_i)$ has length 1. If some point $x\in \gamma_i\cap N_3(c_j)$, then for any $y\in\gamma_i$, we have
$$d(y, N_3(c_j))\leqslant d(x,y)\leqslant\frac12<2\log 2$$
It follows that $\gamma_i$ lies within $N(c_j)$. Either the two punctures are the same or the short geodesic $c_j$ can be isotoped into the cusp of $c_i$. A contradiction.
\end{enumerate}
\end{proof}

Now we estimate the injectivity radius for points in the thick part $\caln$.

\begin{lemma}\label{lemma:inj_radius_outside_nbhds}
The injectivity radius of any point in the thick part of $\Sigma$
is at least 
$$\sinh^{-1}\bracket{\frac14 e^{-\frac{L}4}}>\frac{25}{101}e^{-\frac{L}{4}}$$.
\end{lemma}

\begin{proof} We prove the lemma by contradiction. Let $x$ be a point in the thick part $\Sigma\setminus\caln$ and suppose that the injectivity radius $r_0$ at $x$ satisfies
$$\sinh(r_0)<\frac{1}{4}e^{-\frac{L}{4}}.$$
There exists a homotopically nontrivial simple closed curve $\gamma$ through $x$ with $\ell(\gamma)=2r_0$. $\gamma$ is freely homotopic to either a (unique) simple closed geodesic $\gamma^\prime$, or a puncture $c$ of $\Sigma$.

Suppose that $\gamma$ is freely homotopic to a puncture $c_i$ of $\Sigma$. Let $\widetilde{x}_1$ and $\widetilde{x}_2$ be lifts of $x$ as in Lemma \ref{lemma:injectivity_radius_cusp}. Then $d(\widetilde{x}_1,\widetilde{x}_2)=2r_0<\frac12$ and hence $x\in N_3(c_i)$. Now $x\notin N_2(c_i)$ and $d(x,\partial N(c_i))\leqslant2\log 2+\frac{L}4$. By Lemma \ref{lemma:injectivity_radius_cusp}, we have
$$\sinh(r_0)\geqslant\frac{2}{e^{2\log2+\frac{L}4}}=\frac{1}{2}e^{-\frac{L}4}>\frac14e^{-\frac{L}4}$$

Now suppose that $\gamma$ is freely homotopic to a simple closed geodesic $\gamma^\prime$. Then $\ell(\gamma^\prime)\leqslant 2r_0<e^{-L/4}$ and hence $\gamma^\prime$ is a curve $c_i$ in $\calx$. There are two cases according to whether $c_i\cap \gamma=\emptyset$.

\begin{enumerate}
\item If $c_i\cap\gamma=\emptyset$, then $c_i$ and $\gamma$ co-bound an annulus since they are homotopic and disjoint in $\Sigma$. Since $x\notin N_2(c_i)$, we have $d(x,c_i)\geqslant\log\frac{1}{\ell(c_i)}-\frac{L}4$. By Lemma \ref{lemma:ends_annulus_length}, we have
$$\sinh(r_0)>\frac14e^{d(x,c_i)}\ell(c_i)\geqslant\frac14\exp\bracket{\log\frac{1}{\ell(c_i)}-\frac{L}4}\ell(c_i)=\frac14 e^{-\frac{L}4}.$$
Hence
\begin{equation}
r_0>\sinh^{-1}\bracket{\frac14 e^{-\frac{L}4}}>\frac{100}{101}\cdot\frac14 e^{-\frac{L}4}=\frac{25}{101}e^{-\frac{L}4}
\end{equation}
since 
$$\frac14e^{-L/4}\leqslant\frac{\sqrt{2}-1}4\qquad\text{and}\qquad\sinh^{-1}(t)/t>\frac{100}{101},\quad\forall\,t\in\left(0,\frac{\sqrt{2}-1}4\right].$$

\item
If $c_i\cap\gamma\neq\emptyset$, let $x_0\in c_i\cap\gamma$. If $r_0<\sinh^{-1}(\frac14e^{-{L}/4})<\frac14e^{-{L}/4}\leqslant\frac14(\sqrt{2}-1)<0.11$, then $\ell(c_i)\leqslant\ell(\gamma)<\frac12e^{-{L}/{4}}$, then ${\log{\frac{1}{\ell(c_i)}}-\frac{L}{4}}>\log2>0.69$. Hence 
$$d(c_i,x)\leqslant d(x_0,x)\leqslant\frac{\ell(\gamma)}2\sim0.11<0.69<\log{\frac{1}{\ell(c_i)}}-\frac{L}{4} $$
If follow that $x\in N_2(c_i)$ and we get a contradiction.
\end{enumerate}
\end{proof}

\clearpage
\section{Estimate of self-intersections}\label{sec:estimate}

Let $\Gamma$ be a closed geodesic with Let $L=\ell(\Gamma)\geqslant4\log(1+\sqrt{2})>3.5$ on an orientable, metrically complete hyperbolic surface $\Sigma$ of finite type. Let $\calx$ be the set of short geodesics, and $\caly$ be the set of punctures, as in \S\ref{sec:nbhds}. Suppose $\Gamma$ is represented as a local isometry $f:S^1\to\Sigma$, where $S^1$ is a circle with length $L$. Let $\caln=\bigcup_{c\in\calx\cup\caly}N_2(c)$ be the thin part of $\Sigma$. Let $\cald\subset\Sigma$ be the set of self-intersection points of $\Gamma$, that is,
$$\cald=\set{x\in\Sigma\,:\, \exists\,s,t\in S^1, f(s)=f(t)=x, s\neq t}$$
The \emph{self-intersection number} of $\Gamma$ is defined as
\begin{equation}\label{eqn:self_intersection_number_defn}
\abs{\Gamma\cap\Gamma}:=\sum_{x\in \cald}\binom{\# f^{-1}(x)}{2}
\end{equation}
(See \cite[Section 2]{B2803} for a further discussion on the definition of the self-intersection number.)


\begin{prop}\label{prop:no_intersection_with_N1}
$\Gamma$ is disjoint from the $N_1$ neighborhood of short geodesics in $\calx$ and cusps in $\caly$, that is, 
$\Gamma\cap{\bigcup_{c\in\calx\cup\caly}N_1(c)}=\emptyset$.
\end{prop}

\begin{proof}
Suppose that there exists a point $x_0\in N_1(c_i)\cap\Gamma$ for some $c_i\in\calx$. Since $\ell(\Gamma)=L$, we have $d(x_0,x)\leqslant\frac L2$ for any point $x\in\Gamma$ and
$$d(x,c_i)\leqslant d(x_0,c_i)+\frac{L}{2}<{\log\frac{1}{\ell(c_i)}}<w(\ell(c_i))$$
Hence $x\in{N(c_i)}$ and $\Gamma\subset N(c_i)$. A contradiction. The case for cusps is similar.
\end{proof}

\begin{prop}\label{prop:intersection_with_N2}
There exists at most one $c_i\in\calx\cup\caly$ such that $\Gamma\cap{N_2(c_i)}\neq\emptyset$. $\Gamma\cap{N_2(c_i)}$ is an arc when $\Gamma\cap{N_2(c_i)}\neq\emptyset$.
\end{prop}

\begin{proof}
If there exist distinct $c_i,c_j\in\calx\cup\caly$ such that $\Gamma\cap{N_2(c_i)}\neq\emptyset$ and $\Gamma\cap{N_2(c_j)}\neq\emptyset$. Pick $x_i\in{\Gamma\cap{N_2(c_i)}}$ and let $s=f^{-1}(x_i)$. Consider the arc
$$I_i=\set{t\in S^1\,:\,d_{S^1}(s,t)<d(x_i,\partial N_3(c_i))}$$
in the circle $S^1$ of length $L$. 
Since $d(x_i,\partial N_3(c_i))>\frac{L}4$, we have
$$\ell(I_i)
>\frac{L}{2}$$
Hence $f(I_i)\subset N_3(c_i)$ and $\ell(\Gamma\cap N_3(c_i))>\frac{L}{2}$. Similarly, $\ell(\Gamma\cap N_3(c_j))>\frac{L}2$. $\Gamma\cap N_3(c_i)$ and $\Gamma\cap N_3(c_j)$ are disjoint by Lemma \ref{lemma:nbhd_N3_disjoint}. Hence $\ell(\Gamma)>L$ and we get a contradiction. Therefore there exists at most one $c_i\in\calx\cup\caly$ such that $\Gamma\cap N_2(c_i)\neq\emptyset$.

$\Gamma$ cannot entirely lie within $N_3(c_i)$, and $f^{-1}(\Gamma\cap{N_3(c_i)})$ is a collection of arcs. The argument in previous paragraph implies that such an arc has length bigger than $\frac{L}2$ if it intersects $N_2(c_i)$. Hence there exists at most one arc, say $\gamma$, in $\Gamma\cap{N_3(c_i)}$ that can possibly intersect $N_2(c_i)$.

If $c_i$ is a cusp, then the lift of $\gamma$ in $\bbh^2$ can intersect a horocycle at most twice and hence $\gamma\cap N_2(c_i)$ is an arc. If $c_i\in\calx$, the distance from points on $\gamma$ to a lift of $c_i$ has no local maxima, and $\gamma\cap N_2(c_i)$ is an arc. It follows that $\Gamma\cap N_2(c_i)$ is an arc if nonempty.


\end{proof}

Let $\Gamma_2=\Gamma\cap\caln$ be the part of $\Gamma$ in the thin part $\caln$ and $\Gamma_1=\Gamma\setminus\Gamma_2$ be the part of $\Gamma$ in $\Sigma\setminus\caln$. Note that $\Gamma_2$ could be empty. Let $L_1=\ell(\Gamma_1)$ and $L_2=\ell(\Gamma_2)$, then $L=L_1+L_2$. By Proposition \ref{prop:intersection_with_N2}, $f^{-1}(\Gamma_1)$ and $f^{-1}(\Gamma_2)$ are both connected. The set $\cald$ of self-intersection points of $\Gamma$ consists of the self-intersection points of $\Gamma_1$ and $\Gamma_2$ since $\Gamma_1\cap\Gamma_2=\emptyset$. Let $\cald_1$ and $\cald_2$ be the sets of self-intersection points of $\Gamma_1$ and $\Gamma_2$ respectively. Then $\cald=\cald_1\cup\cald_2$ and
\begin{equation}\label{eqn:split_intersection_number}
\abs{\Gamma\cap\Gamma}=\abs{\Gamma_1\cap\Gamma_1}+\abs{\Gamma_2\cap\Gamma_2}
\end{equation}
(The definition \eqref{eqn:self_intersection_number_defn} of self-intersection numbers can be generalized for geodesic segments.)
By Propositions \ref{prop:no_intersection_with_N1} and \ref{prop:intersection_with_N2}, there are three possibilites
\begin{enumerate}
\item \emph{Simple case}: $\Gamma_2=\emptyset$ and $\Gamma_1=\Gamma$.
\item \emph{General case}: $\Gamma$ intersects $N_2(c_i)$ for some $c_i\in\calx\cup\caly$ and $\Gamma$ intersects only one boundary of $\partial{N_2(c_i)}$. In this case, as $x$ moves along $\Gamma_2$, $d(x,c_i)$ for $c_i\in\calx$ and $d(x,\partial N_1(c_i))$ for $c_i\in\caly$ decrease first and then increase. The general case is illustrated in Figure \ref{fig:general_case}.
\begin{figure}[htbp]
\centering
\tikzexternaldisable	
\begin{tikzpicture}[declare function={
	R = 42;
	f(\x) = R+.5-sqrt(R*R-\x*\x);
	ratio = 3;
}]
\def\x{2.5}
\draw [thick] (-4*\x,{f(4*\x)}) arc ({270-asin(4*\x/R)}:270:{R} and {R});
\draw [thick] (-4*\x,{-f(4*\x)}) arc ({90+asin(4*\x/R)}:90:{R} and {R});

\draw [thick, color=red] (-4*\x,0) ellipse ({f(4*\x)/ratio} and {f(4*\x)});

\foreach \a/\b in {-2.5*\x/orange, -1*\x/cyan, 0/blue} {
	\draw [color=\b, thick] (\a,{-f(\a)}) arc (-90:90:{f(\a)/ratio} and {f(\a)});
	\draw [dashed, color=\b, thick] (\a,{f(\a)}) arc (90:270:{f(\a)/ratio} and {f(\a)});
}

\draw [thick, decoration={
	brace, mirror, raise=0.1cm
	}, decorate
] (-1*\x, {-f(1*\x)}) -- (0,{-f(1*\x)}) 
	node [pos=0.5,anchor=north,yshift=-0.1cm] {$N_1(c_i)$}; 

\draw [thick, decoration={
	brace, mirror, raise=0.2cm
	}, decorate
] (-4*\x, {-f(4*\x)}) -- (0,{-f(3*\x)}) 
	node [pos=0.5,anchor=north,yshift=-0.2cm] {$N_3(c_i)$}; 

\draw [thick, decoration={
	brace, raise=0.1cm
	}, decorate
] (-2.5*\x,{f(2.5*\x)}) -- (0,{f(2*\x)}) 
	node [pos=0.5,anchor=south,yshift=0.1cm] {$N_2(c_i)$}; 

\draw[smooth, color=blue, thick] (-5*\x,0.4*\x)
  to [start angle=-20, next angle=0] ($(-4*\x,0)+(160:{f(4*\x)/ratio} and {f(4*\x)})$)
  to [next angle=40] ($(-4*\x,0)+(40:{f(4*\x)/ratio} and {f(4*\x)})$) coordinate(x);
\draw[smooth, color=blue, thick, dashed] (x) 
	to [start angle=40, next angle=-5] (-9,{f(9)}) coordinate (x);
\draw[smooth, color=blue, thick] (x) 
	to [start angle=-5, next angle=280] (-8.3,0) 
	to [next angle=5] (-7.8,{-f(7.8)}) coordinate (x);
\draw[smooth, color=blue, thick, dashed] (x) 
	to [start angle=5, next angle=80] (-7.35,0) 
	to [next angle=-5] (-6.7,{f(6.7)}) coordinate (x);
\draw[smooth, color=blue, thick] (x) 
	to [start angle=-5, next angle=280] (-6.1,0) 
	to [next angle=5] (-5.7,{-f(5.7)}) coordinate (x);
\draw[smooth, color=blue, thick, dashed] (x) 
	to [start angle=5, next angle=80] (-5.25,0) 
	to [next angle=-5] (-4.8,{f(4.8)}) coordinate (x);
\draw[smooth, color=blue, thick] (x) 
	to [start angle=-5, next angle=280] (-4.3,0) 
	to [next angle=5] (-4,{-f(4)}) coordinate (x);
\draw[smooth, color=blue, thick, dashed] (x) 
	to [start angle=5, next angle=80] (-3.65,0) 
	to [next angle=-5] (-3.3,{f(3.3)}) coordinate (x);
\draw[smooth, color=blue, thick] (x) 
	to [start angle=-5, next angle=270] (-2.8,0) 
	to [next angle=185] (-3.3,{-f(3.3)}) coordinate (x);
\draw[smooth, color=blue, thick, dashed] (x) 
	to [start angle=185, next angle=100] (-3.65,0) 
	to [next angle=175] (-4,{f(4)}) coordinate (x);
\draw[smooth, color=blue, thick] (x) 
	to [start angle=175, next angle=260] (-4.3,0) 
	to [next angle=185] (-4.8,{-f(4.8)}) coordinate (x);
\draw[smooth, color=blue, thick, dashed] (x) 
	to [start angle=185, next angle=100] (-5.25,0) 
	to [next angle=175] (-5.7,{f(5.7)}) coordinate (x);
\draw[smooth, color=blue, thick] (x) 
	to [start angle=175, next angle=260] (-6.1,0) 
	to [next angle=185] (-6.7,{-f(6.7)}) coordinate (x);
\draw[smooth, color=blue, thick, dashed] (x) 
	to [start angle=185, next angle=100] (-7.35,0) 
	to [next angle=175] (-7.8,{f(7.8)}) coordinate (x);
\draw[smooth, color=blue, thick] (x) 
	to [start angle=175, next angle=260] (-8.3,0) 
	to [next angle=185] (-9,{-f(9)}) coordinate (x);
\draw[smooth, color=blue, thick, dashed] (x) 
	to [start angle=185, next angle=140] ($(-4*\x,0)+(-40:{f(4*\x)/ratio} and {f(4*\x)})$) coordinate(x);
\draw[smooth, color=blue, thick] (x) 
	to [start angle=140, next angle=180] ($(-4*\x,0)+(200:{f(4*\x)/ratio} and {f(4*\x)})$)
	to [next angle=200] (-5*\x,-0.4*\x);

  			
\path ({f(0)/ratio},0) node[circle, inner sep=1pt, label={[label distance=.4em, anchor=center]0:$c_i$}]{};
  			
\end{tikzpicture}
\caption{\label{fig:general_case}
The general case}
\end{figure}
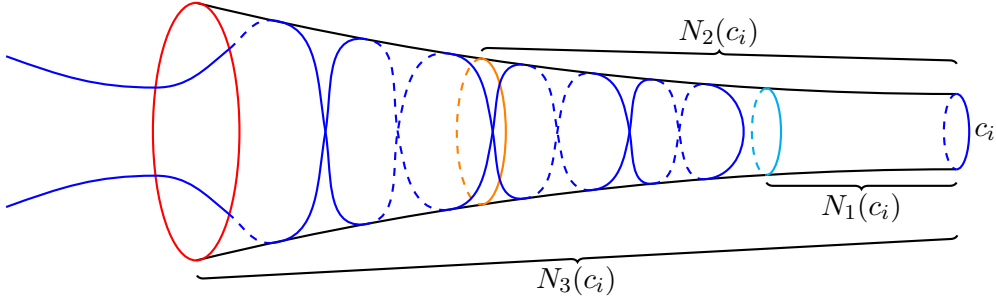

\item \emph{Special case}: $\Gamma$ intersects some $N_2(c_i)$ and $\Gamma$ intersects only both boundaries of $\partial{N_2(c_i)}$. In this case $N_1(c_i)=\emptyset$ by Proposition \ref{prop:no_intersection_with_N1}, and $\Gamma\cap{N_2(c_i)}$ is an embedded arc in $\Sigma$. Note that the special case does not happen for cusps.
\end{enumerate}

The self-intersection numbers of $\Gamma_1$ and $\Gamma_2$ are estimated in the following two theorems.

\begin{figure}[htbp]
\centering
\tikzexternaldisable
\begin{tikzpicture}[declare function={
	R = 9;
	f(\x) = R+.5-sqrt(R*R-\x*\x);
	ratio = 3;
}]
\def\x{1}
\draw [thick] (-4,{f(4)}) arc ({270-asin(4/R)}:{270+asin(4/R)}:{R} and {R});
\draw [thick] (-4,{-f(4)}) arc ({90+asin(4/R)}:{90-asin(4/R)}:{R} and {R});

\draw [thick, color=red] (-4,0) ellipse ({f(4)/ratio} and {f(4)});

\foreach \a/\b in { -2/orange,  0/blue,  2/orange, 4/red} {
	\draw [color=\b, thick] (\a,{-f(\a)}) arc (-90:90:{f(\a)/ratio} and {f(\a)});
	\draw [dashed, color=\b, thick] (\a,{f(\a)}) arc (90:270:{f(\a)/ratio} and {f(\a)});
}


\draw[smooth, color=blue, thick] (-6,-0.4)
  to [start angle=30, next angle=5] ($(-4*\x,0)+(175:{f(4*\x)/ratio} and {f(4*\x)})$)
	to [next angle=5] ($(-4*\x,0)+(10:{f(4*\x)/ratio} and {f(4*\x)})$) coordinate(x);
\draw[smooth, color=blue, thick, dashed] (x) 
	to [start angle=5, next angle=0] (-0.5,{f(-0.5)}) coordinate (x);
\draw[smooth, color=blue, thick] (x)
  to [start angle=0, next angle=-15] ($(4*\x,0)+(-20:{f(4*\x)/ratio} and {f(4*\x)})$)
  to [next angle=-32] (6, -1.3);

\draw [thick, decoration={
	brace, mirror, raise=0.1cm
	}, decorate
] (-2, {-f(2)}) -- (2,{-f(2)}) 
	node [pos=0.5,anchor=north,yshift=-0.1cm] {$N_2(c_i)$}; 

\draw [thick, decoration={
	brace, raise=0.1cm
	}, decorate
] (-4,{f(4)}) -- (4,{f(4)}) 
	node [pos=0.5,anchor=south,yshift=0.1cm] {$N_3(c_i)$}; 
  			
\path ({f(0)/ratio},0) node[circle, inner sep=1pt, label={[label distance=.4em, anchor=center]0:$c_i$}]{};
  			
\end{tikzpicture}

\caption{\label{fig:special_case}
The special case}
\end{figure}
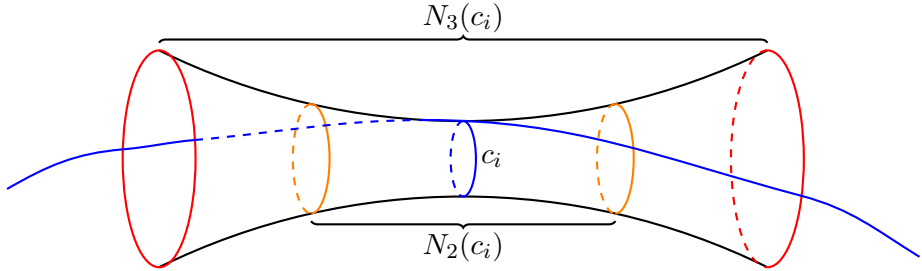

\begin{thm}[Thick part]\label{thm:estimate_thick_part}
\begin{equation}\label{eqn:estimate_thick_part}
\abs{\Gamma_1\cap\Gamma_1}\leqslant\frac{101}{2\cdot25}L_1e^{\frac{L}4}+\frac{101^2}{2\cdot25^2}L_1^2e^{\frac{L}2}
\end{equation}
\end{thm}

\begin{proof}
Let $\calp_1$ be the preimages of $\cald_1$ with multiplicity, that is, if $x\in\cald_1$ has $n$ preimages, any $s\in f^{-1}(x)$ has multiplicity $(n-1)$ in $\calp_1$. Then
$$\#\calp_1=2\abs{\Gamma_1\cap\Gamma_1}.$$
Divide $\Gamma_1$ into $M:=\gauss{\frac{101}{25} L_1e^{\frac{L}{4}}}+1$ short closed segments with length less than $\frac{L_1}{\frac{101}{25} L_1e^{\frac{L}{4}}}=\frac{25}{101}e^{-\frac{L}{4}}$.
Let $S$ be the set of the segments. Since $\calp_1$ is finite, $S$ can be chosen such that the endpoints of the segments do not contain points in $\calp_1$. 
$f(I_\alpha)\subset\Gamma_1$ is in the thick part and $\ell(I_\alpha)<\frac{25}{101}e^{-\frac{L}{4}}$, hence $f(I_\alpha)$ has no self-intersections by Lemma \ref{lemma:inj_radius_outside_nbhds}. 

We claim that any two distinct $I_\alpha, I_\beta\in S$ have at most one intersection. If there exist $s_1,s_2\in I_\alpha$ ($s_1\neq s_2$) and $t_1,t_2\in I_\beta$ such that $f(s_1)=f(t_1)$ and $f(s_2)=f(t_2)$. Let $\gamma_\alpha$ be the segment in $I_\alpha$ between $t_1$ and $t_2$, and $\gamma_\beta$ be the segment in $I_\beta$ between $s_1$ and $s_2$. Then $f(\gamma_\alpha)$ and $f(\gamma_\beta)$ are two distinct geodesics between $f(s_1)$ and $f(s_2)$. Since the injectivity radius at $f(s_1)=f(t_1)$ is at least $\frac{25}{101}e^{-L/4}$ and $\ell(f(\gamma_\alpha)),\ell(f(\gamma_\beta))<\frac{25}{101}e^{-\frac{L}{4}}$, the existence of $f(\gamma_\alpha)$ and $f(\gamma_\beta)$ contradicts the uniqueness of geodesics in $\bbh^2$.

Since any two distinct $I_\alpha, I_\beta\in S$ contribute at most 1 to $\abs{\Gamma_1\cap\Gamma_1}$, we have
$$\abs{\Gamma_1\cap\Gamma_1}\leqslant\frac12M(M-1)\leqslant \frac{1}{2}\bracket{1+\frac{101}{25} L_1e^{\frac{L}{4}}}\cdot \frac{101}{25} L_1e^{\frac{L}{4}}=\frac{101}{2\cdot25}L_1e^{\frac{L}4}+\frac{101^2}{2\cdot25^2}L_1^2e^{\frac{L}2}$$

\end{proof}

\begin{thm}[Thin part]\label{thm:estimate_thin_part}
\begin{equation}\label{eqn:estimate_thin_part}
\abs{\Gamma_2\cap\Gamma_2}\leqslant\frac{25}{12}L_2e^{\frac{L}2}.
\end{equation}
\end{thm}

\begin{proof}
In the simple and special cases, either $\Gamma_2$ is empty or $\Gamma_2$ has no self-intersections points and the theorem is trivial. In the general case, the geodesic $\Gamma_2$ meets only one side of $N_2(c_i)$ for some $c_i\in\calx\cup\caly$. $\Gamma_2$ has only double points and hence $\abs{\Gamma_2\cap\Gamma_2}=\# \cald_2$.

Suppose that $c_i\in\calx$. The universal covering $\bbh^2\to\Sigma$ restricts to a universal covering $p:\Omega\to N_2(c_i)$ of $N_2(c_i)$. We may assume that $p^{-1}(c_i)$ is the horizontal line $\bbh^1\subset\bbh^2$ (in the Poincar\'e disk model). Let $\widetilde{\Gamma}$ be a lift of $\Gamma_2$. Assume $x_0\in \widetilde{\Gamma}$ is a point that has smallest distance to $\bbh^1$. Let $\phi:[-A,B]\to\widetilde{\Gamma}$ be an isometry so that $\phi(0)=x_0$. The function $d(\phi(s),\bbh^1)$ 
is first decreasing and then increasing when $s$ increases from $-A$ to $B$ and $d(\phi(s),\bbh^1)=d(\phi(-s),\bbh^1)$. The self intersection points of $\Gamma_2$ are the image of pairs $(-s,s)\in[-A,B]$ such that $p(\phi(-s))=p(\phi(s))$. 
Let $(-s_1,s_1),\cdots,(-s_K,s_K)$ be the pairs with $s_0=0<s_1<s_2<\cdots<s_K$, where $K=\abs{\Gamma_2\cap\Gamma_2}$ is the self-intersection number of $\Gamma_2$. For each $0\leqslant i\leqslant K-1$, $p(\phi([-s_{i+1},-s_i]\cup[s_i,s_{i+1}]))$ is a homotopy nontrivial circle in $N_2(c_i)$. By Lemma \ref{lemma:ends_annulus_length}, we have
$$2(s_{i+1}-s_i)>\frac{12}{25}\ell(c_i)\exp\bracket{\log{\frac{1}{\ell(c_i)}}-\frac{L}{2}}=\frac{12}{25}e^{-\frac{L}{2}}.$$

When $c_i\in\caly$ is a cusp, the proof is similar where the geodesic line $\bbh^1$ is replaced by a horocycle. Suppose that a lift of $c_i$ is $\infty$ in the upper plane model and $\widetilde{\Gamma}$ be a lift of $\Gamma$. Let $x_0\in\widetilde{\Gamma}$ be the point with largest vertical coordinate. The pairs $(-s_1,s_1),\cdots,(-s_K,s_K)$ with $s_0=0<s_1<s_2<\cdots<s_K$ are similarly defined. Each $p(\phi(s_i))\in N_2(c_i)\setminus N_1(c_i)$ and hence 
$$2\log2+\frac{L}4\leqslant d(p(\phi(s_i)), \partial N(c_i))<2\log2+\frac{L}2.$$ 
By Lemma \ref{lemma:injectivity_radius_cusp}, we have
$$2(s_{i+1}-s_i)> 2 \sinh^{-1}\bracket{2\exp\bracket{-2\log2-\frac{L}2}}>\frac{24}{25}e^{-\frac{L}{2}}.$$

Since $\ell(\Gamma_2)=L_2$, we have $\abs{\Gamma_2\cap\Gamma_2}<\frac{25}{12}L_2e^{{L}/{2}}$.


\end{proof}

\begin{proof}[Proof of Theorem \ref{thm:bound_intersection_by_length}]
$\abs{\Gamma\cap\Gamma}=0$ when $L<4\log(1+\sqrt{2})$ by \cite[Lemma 7]{Y2799} and Theorem \ref{thm:bound_intersection_by_length} is trivial in this case. For $L\geqslant4\log(1+\sqrt{2})\sim3.5$, from \eqref{eqn:split_intersection_number}, \eqref{eqn:estimate_thick_part} and \eqref{eqn:estimate_thin_part}, we have $4<\sigma=\frac{101}{25}<\frac{25}{6}$ and 
$$\begin{aligned}
\abs{\Gamma\cap\Gamma}=&\abs{\Gamma_1\cap\Gamma_1}+\abs{\Gamma_2\cap\Gamma_2}\leqslant\frac{101}{2\cdot25}L_1e^{\frac{L}4}+\frac{101^2}{2\cdot25^2}L_1^2e^{\frac{L}2}+\frac{25}{12}L_2e^{\frac{L}2}\\
=&\frac{101}{2\cdot25}L_1e^{\frac{L}4}+\frac{101^2}{2\cdot25^2}L_1^2e^{\frac{L}2}+\frac{25}{12}(L-L_1)e^{\frac{L}2}
\leqslant\frac{101^2}{2\cdot25^2}L_1^2e^{\frac{L}2}+\frac{25}{12}Le^{\frac{L}2}\\
\leqslant&\bracket{\frac{101^2}{2\cdot 25^2}L^2+\frac{25}{12}L}e^{\frac{L}2}<9L^2e^{\frac{L}2}
\end{aligned}$$ 
(The last inequality holds when $\displaystyle{L>2.5}$.) Hence Theorem \ref{thm:bound_intersection_by_length} holds for all $L$.
\end{proof}

\begin{proof}[Proof of Corollary \ref{thm:lower_bound}]
For any $\varepsilon>0$, let $\displaystyle{\varepsilon_1=\min\set{\varepsilon,\frac15}}$ and $\displaystyle{k_0:=\exp\bracket{\frac{6}{\varepsilon^2_1}}}$. Let $k>k_0$, then $\displaystyle{\log k>\frac{6}{\varepsilon_1^2}}$.
Let $f(x)=\varepsilon_1 x-4\log 6-4\log x$, then
\begin{align*}
f\bracket{\frac{6}{\varepsilon_1^2}}=&\varepsilon_1\cdot\frac{6}{\varepsilon_1^2}-4\log6-4\log\frac{6}{\varepsilon_1^2}\\
\geqslant&\frac{6}{\varepsilon_1}-8\log6-8\cdot\frac{\log5}{4}\bracket{\frac{1}{\varepsilon_1}-1} & \bracket{\text{since }\log x\leqslant\frac{\log5}{4}(x-1)\text{ for } x\geqslant5}\\
=&\frac{6-\log 25}{\varepsilon_1}-8\log 6+2\log 5\\
\geqslant&5\cdot(6-\log25)-8\log 6+2\log5 >0.
\end{align*}
Since $f^\prime(x)=\varepsilon_1-\frac4x>0$ for $x>\frac{6}{\varepsilon_1^2}$, we have
\begin{equation}\label{eqn:lower_bound_1}
\frac\varepsilon2\log k\geqslant\frac{\varepsilon_1}2\log k >2\log 6 +2\log \log k
\end{equation}
Suppose that there exists a $k$-geodesic with length $L<(2-\varepsilon)\log k$. By Theorem \ref{thm:bound_intersection_by_length}, we have $k<9L^2e^{L/2}$. Hence
\begin{align*}
\log k<&2\log 3+2\log L+\frac{L}2<2\log3+2\log(2-\varepsilon)+2\log\log k+\frac12(2-\varepsilon)\log k\\<&2\log 6 + 2\log\log k+\log k-\frac{\varepsilon}2\log k
\end{align*}
This contradicts \eqref{eqn:lower_bound_1}. Hence any $k$-geodesic $\omega$ with $k>k_0$ has length $\ell(\omega)>(2-\varepsilon)\log k$.
\end{proof}


\end{document}